\documentclass{amsart}

\usepackage{amsmath,amsfonts,amsthm,latexsym,amssymb,mathrsfs,setspace,enumitem,color,cite,graphicx,epsf,tikz-cd}
\usepackage{fullpage}
\usepackage{lineno}

\newtheorem{thm}{Theorem}[section]
\newtheorem{lem}[thm]{Lemma}
\newtheorem{prop}[thm]{Proposition}
\newtheorem{cor}[thm]{Corollary}
\numberwithin{equation}{section}
\numberwithin{thm}{section}

\theoremstyle{definition}

\newtheorem{defin}[thm]{Definition}

\newcommand{\rat}{\mathbb Q}

\newcommand{\alg}{{\overline\rat}}
\newcommand{\algt}{{\alg^{\times}}}

\newcommand{\nat}{\mathbb N}

\newcommand{\gal}{\mathrm{Gal}}

\newcommand{\tors}{\mathrm{tors}}

\newcommand{\con}{\mathrm{con}}

\newcommand{\Add}{\mathrm{Add}}
\newcommand{\Mult}{\mathrm{Mult}}
\newcommand{\directlimit}{\mathbb A}
\newcommand{\crazyadele}{\mathbb V}
\newcommand{\I}{\mathcal I}

\title{Direct Limits of Ad\`ele Rings and Their Completions}

 \author{James P. Kelly}
     \address{Christopher Newport University, Department of Mathematics, 1 Avenue of the Arts, Newport News, VA 23606, USA}
     \email{james.kelly@cnu.edu}
     
  \author{Charles L. Samuels}
     \address{Christopher Newport University, Department of Mathematics, 1 Avenue of the Arts, Newport News, VA 23606, USA}
     \email{charles.samuels@cnu.edu}

\date{December 19, 2019}


\begin{document}

\begin{abstract}
	The ad\`ele ring $\mathbb A_K$ of a global field $K$ is a locally compact, metrizable topological ring which is complete with respect to any invariant metric on $\mathbb A_K$.  
	For a fixed global field $F$ and a possibly infinite algebraic extension $E/F$, there is a natural partial ordering on $\{\mathbb A_K:F\subseteq K\subseteq E\}$.  Therefore, we may form the direct limit
	\[
		\directlimit_E = \varinjlim \mathbb A_K
	\]
	which provides one possible generalization of ad\`ele rings to arbitrary algebraic extensions $E/F$.  In the case where $E/F$ is Galois, 
	we define an alternate generalization of the ad\`eles, denoted $\overline{\crazyadele}_E$, to be a certain metrizable topological ring of continuous functions on the set of places of $E$.   
	We show that $\overline{\crazyadele}_E$ is isomorphic to the completion of $\directlimit_E$ with respect to any invariant metric
	and use this isomorphism to establish several topological properties of $\directlimit_E$.
\end{abstract}

\maketitle

\section{Introduction}

Suppose that $K$ is a global field and $Y_K$ is the set of all places of $K$.  For each $v\in Y_K$ we write $K_v$ to denote the completion of $K$ with respect to $v$.
If $v$ is a nonarchimedean place of $K$, we define the {\it ring of $v$-adic integers} by $\mathcal O_v = \{\alpha\in K_v: |\alpha|_v \leq 1\}$
and note that this definition does not depend on the choice of absolute value from $v$.   Moreover, $\mathcal O_v$ forms a ring which is open and compact in $K_v$.  A point 
\begin{equation} \label{BasicAdele}
	\mathfrak a = (a_v)_{v\in Y_K}\in \prod_{v\in Y_K}{K_v}
\end{equation}
is called an {\it ad\`ele of $K$} if there exists a finite subset $S\subseteq Y_K$ such that $a_v\in \mathcal O_v$ for all $v\not\in S$.
Since there are only finitely many archimedean places of $K$, this definition is undisturbed by the fact that $\mathcal O_v$ is undefined at such places.
By defining addition and multiplication coordinatewise, the set of all ad\`eles of $K$ forms a ring $\mathbb A_K$ called the {\it ad\`ele ring of $K$}. 
We define a topology on $\mathbb A_K$ by taking as a basis the sets of the form 
\begin{equation*}
	\prod_{v\in Y_K} U_v \subseteq \prod_{v\in Y_K} K_v
\end{equation*}
which satisfy the following properties:
\begin{itemize}
	\item There exists a finite subset $S\subseteq Y_K$ such that $U_v = \mathcal O_v$ for all $v\not\in S$
	\item $U_v$ is open in $K_v$ for all $v\in Y_K$.
\end{itemize}
It is straightforward to verify that these sets do indeed form a basis for a topology on $\mathbb A_K$, and moreover, $\mathbb A_K$ forms a topological ring.

Suppose now that $R$ is a metrizable topological ring.  The Birkhoff-Kakutani Theorem \cite{Birkhoff,Kakutani} implies that the topology of $R$
is induced by an {\it invariant metric} $d$, i.e., a metric $d$ such that $d(x,y) = d(x+z,y+z)$ for all $x,y,z\in R$.  Given a sequence $\{\alpha_n\}_{n=1}^\infty$ in $R$, 
the following conditions are equivalent:
\begin{enumerate}[label={(\roman*)}]
	\item $\{\alpha_n\}_{n=1}^\infty$ is Cauchy with respect to $d$
	\item\label{NbhdCauchy} For every open neighborhood $U$ of $0$ there exists $N\in \nat$ such that $\alpha_m - \alpha_n \in U$ for all $m,n\geq N$.
\end{enumerate}
In view of these observations, the definitions of Cauchy, complete and completion with respect to $d$ are independent of the invariant metric $d$.
Without any ambiguity, we may now write $\overline R$ to denote the completion of $R$ with respect to any invariant metric.
By applying the Birkhoff-Kakutani Theorem and \ref{NbhdCauchy}, one easily verifies the following properties of ad\`ele rings.

\begin{prop}\label{ClassicalAdeles}
	If $K$ is a global field then $\mathbb A_K$ is a metrizable topological ring which is complete with respect to any invariant metric on $\mathbb A_K$.
\end{prop}

Proposition \ref{ClassicalAdeles} can be established directly in a fairly straightforward manner.  However, should the reader seek a proof of Proposition \ref{ClassicalAdeles}, 
we will establish these properties for a more general object later in this article (see Theorem \ref{Main}\ref{TopologicalRingMain}).
It is also worth noting that $\mathbb A_K$ is locally compact, and as a result, it has a Haar measure and is amenable to methods of harmonic analysis.  Although this is a
crucial feature of ad\`ele rings having important applications in number theory, the properties described in Proposition \ref{ClassicalAdeles} are most relevant to our goals.
 
Suppose now that $L$ is a finite extension of $K$.  If $w$ is a place of $L$ and $|\ |$ is an absolute value from $w$, then the restriction of $|\ |$ to $K$ defines a place $v$ of $K$ which is independent of $|\ |$. 
In this case, we say that $w$ {\it divides} $v$ and we write $Y(L/K,v)$ to denote the set of all places of $L$ which divide $v$.  It is well-known that $Y(L/K,v)$ is a non-empty finite set for all $v\in Y_K$.
If $\mathfrak a\in \mathbb A_K$ has the form \eqref{BasicAdele} then there exists a unique point 
\begin{equation*}
	\mathfrak b = (b_w)_{w\in Y_L} \in \mathbb A_L
\end{equation*}
such that $b_w = a_v$ for all $w\in Y(L/K,v)$.
We say that $\mathfrak b$ is the {\it conorm of $\mathfrak a$ from $K$ to $L$} and we write $ \mathfrak b = \con_{L/K}(\mathfrak a)$.  In this way, we interpret the conorm as a map 
$\con_{L/K}:\mathbb A_K\to \mathbb A_L$.  This definition of $\con_{L/K}$ as well as an alternate equivalent definition is provided in \cite[Ch. II, \S 19]{CasselsFrohlich}.  
One easily checks that $\con_{L/K}$ forms a topological ring isomorphism of $\mathbb A_K$ onto $\con_{L/K}(\mathbb A_K)$ with the subspace topology.   Moreover, the diagram
\begin{center}
	\begin{tikzcd}[row sep=huge, column sep = huge]
		\mathbb A_L \arrow[r,"\con_{M/L}"] & \mathbb A_M \\
		\mathbb A_K \arrow[u,"\con_{L/K}"] \arrow[ur,"\con_{M/K}",swap]
	\end{tikzcd}
\end{center}
commutes for all global fields $K\subseteq L\subseteq M$.  In view of these observations, it is common to identify $\mathbb A_K$ with $\con_{L/K}(\mathbb A_K)$.  This identification permits us to write
$\mathbb A_K\subseteq \mathbb A_L$ whenenver $K\subseteq L$ and to refer to the conorm as an inclusion map.

We now fix a global field $F$ for the remainder of this article. The reader may feel free to assume that $F=\rat$ without sacrificing the generality of our results.  
Indeed, had we written the theorems and proofs only for that special case, no substantive modifications would be required to generalize to the situation where $F$ is an arbitrary global field.
Regardless, many of the objects we consider depend on $F$ even though we shall often suppress this dependency in our notation.  

For each (possibly infinite) Galois extension $E/F$ we set 
\begin{equation*}
	\I_E = \{K\subseteq E: K/F\mbox{ finite Galois}\}.
\end{equation*}
Under these assumptions, $\langle \mathbb A_K,\con_{L/K} \rangle$ defines a direct system over $\I_E$ and we may form the direct limit
\begin{equation} \label{DirectLim}
	\directlimit_E := \varinjlim \mathbb A_K
\end{equation}
equipping $\directlimit_E$ with the usual final topology.  For readers unfamiliar with the right hand side of \eqref{DirectLim}, it behaves much like a union of the $\mathbb A_K$.
Indeed, interpreting the conorm as an inclusion map as in our earlier remarks, we may alternatively write 
\begin{equation*}
	\directlimit_E = \bigcup_{K\in \I_E}\mathbb A_K.
\end{equation*}
There is a ring structure on $\directlimit_E$ which causes $\directlimit_E$ to be a topological ring and causes the canonical inclusion maps $\mathbb A_K\to \directlimit_E$ to be topological ring homomorphisms.  
If $E/F$ is a finite Galois extension, then $E$ is a global field and $\directlimit_E$ is equal to the usual ad\`ele ring of $E$ as defined earlier in the introduction.
Therefore, $\directlimit_E$ is one possible generalization of the ad\`eles to arbitrary Galois extensions of global fields.  Rogawski \cite[Ch. 3]{Rogawski} refers to $\directlimit_E$ as the ad\`ele ring of $E$.

If $E/F$ is infinite then $\directlimit_E$ has been only sparingly studied in the literature (Rogawski \cite[Ch. 3]{Rogawski} is the only example of which we are aware), and as such, very little is known 
about its topological and algebraic properties.  For instance, one natural problem is to determine whether the properties described in Proposition \ref{ClassicalAdeles} hold for $\mathbb A_E$.
Although it is straightforward to use the Birkhoff-Kakutani Theorem to prove that $\directlimit_E$ is metrizable, it appears much more difficult to provide meaningful information on its completion with respect to an invariant metric.
The primary goal of this article is to recognize $\overline{\directlimit}_E$ as a space of continuous functions.

We regard our work to be a direct analog of a problem encountered by Allcock and Vaaler \cite{AllcockVaaler} which deals with normed vector spaces.
Letting $\alg$ denote an algebraic closure of $\rat$ and $h$ the Weil height, they studied the vector space $\mathscr V := \algt/\alg^\times_\tors$ over $\rat$ equipped with the norm given by
\begin{equation*}
	\alpha\mapsto 2h(\alpha).
\end{equation*}
Allcock and Vaaler's stated objective was to determine the completion of $\mathscr V$ with respect to this norm.  They accomplished this task by defining the set $Y_E$ of places of $E$,
equipping $Y_E$ with a topology, and creating a measure $\rho$ on the Borel sets $\mathcal B$ of $Y_{E}$.  
Then \cite[Theorem 1]{AllcockVaaler} and its surrounding remarks described an isometric isomorphism from the completion of ${\mathscr V}$ to 
\begin{equation*}
	\mathcal X:= \left\{ F\in L^1(Y_{\alg},\mathcal B,\rho): \int_{Y_{\alg}} F(y)\mathrm{d}\rho(y) = 0\right\}.
\end{equation*}
In other words, we may now recognize $\mathscr V$ as a space of integrable functions.
	 
Returning to our study of $\directlimit_E$, we will address our goal by constructing an alternate generalization of the ad\`eles, denoted $\overline{\crazyadele}_E$.
Roughly speaking, $\overline{\crazyadele}_E$ is a space of continuous functions on the set of places of $E$, and our main results (Theorem \ref{Main} and Corollary \ref{DirectLimitCompletionMain})
show it to be isomorphic to $\overline{\directlimit}_E$.   
These results are a direct analog of \cite[Theorem 1]{AllcockVaaler} with $\directlimit_E$ playing the role of $\mathscr V$, $\overline{\crazyadele}_E$
playing the role of $\mathcal X$, and continuous functions playing the role of integrable functions.   
As an added benefit, we will be able to use $\overline{\crazyadele}_E$ to show that $\directlimit_E$ has empty interior inside of $\overline{\directlimit}_E$. 
We provide the formal definition of $\overline{\crazyadele}_E$ in Section \ref{Construction} and defer all proofs to Sections 3-7.

\section{An Alternate Generalization of Ad\`ele Rings} \label{Construction}

\subsection{Overview and Intuition}
Since the formal definition of $\overline{\crazyadele}_E$ is quite involved, we find it worthwhile to provide the reader with an intuitive understanding before proceeding with its details.  
Borrowing our notation for global fields, we write $Y_E$ for the set of places of $E$, $E_y$ for the completion of $E$ with respect to a place $y\in Y_E$, and 
$\mathcal O_y = \{\alpha\in E_y: |\alpha|_y \leq 1\}$.  We wish for an {\it ad\`ele of $E$} to be a point
\begin{equation*}
	\mathfrak a = (a_y)_{y\in Y_E} \in \prod_{y\in Y_E} E_y
\end{equation*}
satisfying the following conditions:
\begin{enumerate}[label={(a.\arabic*)}]
	\item\label{BadUnitDisk} There exists a compact set $Z\subseteq Y_E$ such that $a_y\in \mathcal O_y$ for all $y\in Y_E\setminus Z$.
	\item\label{BadContinuous} The map $y\mapsto a_y$ is continuous on $Y_E$.
\end{enumerate}
Of course, this definition does not make sense without further information.  For instance, both conditions require a topology on $Y_E$, and as of this moment, we have not provided one.
If we are able to topologize $Y_E$ in such a way that $Y_E$ is discrete whenever $E/F$ is finite, then we would have defined a direct 
generalization of the classical definition of ad\`ele.  After all, if $Y_E$ is discrete then compactness is equivalent to finiteness and every map having $Y_E$ as its domain is continuous.

The work of Allcock and Vaaler \cite{AllcockVaaler} can help us equip $Y_E$ with such a topology, but even then, our proposed definition still fails.  Since $a_y$ belongs to different completions of $E$ as 
$y$ varies through $Y_E$, the rule described in \ref{BadContinuous} does not create a well-defined map.  In order to resolve this problem, we will
need a system of maps among the various completions of $E$, and moreover, our definition of ad\`ele must be independent of that choice of system.  All of this must be done in a way that causes
$\overline{\crazyadele}_E$ to be isomorphic to $\overline{\directlimit}_E$.

\subsection{Absolute Values and Places of $E$}
We are now prepared to move forward with our formal definition of $\overline{\crazyadele}_E$ which, for reasons outlined above, borrows various aspects of \cite{AllcockVaaler}.
If $K\in \I_E$ and $v$ is a place of $K$, then we write $Y(E/K,v)$ to denote the set of all places of $E$
which divide $v$.  Allcock and Vaaler used the notation $Y(K,v)$ instead of $Y(E/K,v)$.  Moreover, they defined $Y(E/K,v)$ to be a certain inverse limit of finite discrete sets and later 
identified it with the places of $E$ that divide $v$. The set $Y_E$ was defined to be the union of the sets $Y(E/F,p)$.  We refer the reader to \cite[\S 2]{AllcockVaaler} for the details.

According to \cite[\S 2]{AllcockVaaler}, the collection
\begin{equation*}
	\left\{ Y(E/K,v): K\in \I_E,\ v\in Y_K\right\}
\end{equation*}
is a basis for a totally disconnected, Hausdorff topology on $Y_E$, and moreover, each basis element $Y(E/K,v)$ is compact and non-empty.  If $E/F$ is finite then $Y_E$ is discrete.

For each place $p$ of $F$, we select an absolute value $|\ |_p$ from $p$, and as such, $|\ |_p$ extends to a unique absolute value on the completion $F_p$.
For every $K\in \I_E$ and $v\in Y(K/F,p)$, there is a unique extension of $|\ |_p$ to $K_v$, denoted $|\ |_v$ (see \cite[Prop. 2.2]{Lang}). 
Still utilizing the observations of \cite[\S 2]{AllcockVaaler}, each place $y\in Y(E/F,p)$ may be used to define a unique absolute value $|\ |_y$ on $E$ satisfying the following property:
\begin{equation*}
	y\in Y(E/K,v)\implies |\alpha|_y = |\alpha|_v\mbox{ for all } \alpha \in K.
\end{equation*}
In other words, not only does $y$ divide $v$, but the specific absolute value $|\ |_y$ agrees with $|\ |_v$ when restricted to $K$.  Moreover, $|\ |_y$ extends to a unique absolute
value on $E_y$ whose restriction to $K_v$ is equal to $|\ |_v$.

We assume now that $\gal(E/F)$ is equipped with the Krull topology as in \cite[Ch. IV, \S 1]{Neukirch}.
If $K\in \I_E$ then according to \cite[\S 3]{AllcockVaaler}, there is a well-defined action of the normal subgroup $\gal(E/K)$ on $Y(E/K,v)$ which satisfies the identity
$|\sigma(\alpha) |_{\sigma(y)} = |\alpha |_y$ for all $\alpha\in E$, $\sigma\in \gal(E/K)$ and $y\in Y(E/K,v)$.  Moreover, this action is transitive and the map
\begin{equation*}
	\gal(E/K) \times Y(E/K,v)\to Y(E/K,v),\quad (\sigma,y)\mapsto \sigma(y)
\end{equation*}
is continuous (see \cite[Ch. II, Prop. 9.1]{Neukirch} and \cite[Lemma 3]{AllcockVaaler}, respectively).
Thus, each element $\sigma\in \gal(E/K)$ extends to a map $E_y\to E_{\sigma(y)}$ which satisfies $ |\sigma(\alpha) |_{\sigma(y)} = |\alpha |_y$ for all $\alpha\in E_y$.  This means that $\sigma$ defines an 
isometric isomorphism from $E_y$ to $E_{\sigma(y)}$.

While the isometric isomorphisms provided by \cite{AllcockVaaler} play an important role, they are not sufficient on their own to construct $\overline{\crazyadele}_E$. 
Our definition of $\overline{\crazyadele}_E$ will require those maps to be selected in a way which satisfies certain algebraic and topological properties.  Fortunately, such maps must always exist.

\begin{thm} \label{TransitionExistence}
	Suppose $E/F$ is a Galois extension and $K\in \I_E$.  If $v\in Y_K$ then there exists a map $\lambda:Y(E/K,v)\times Y(E/K,v)\to \gal(E/K)$ satisfying the following conditions:
	\begin{enumerate}[label={(D.\arabic*)}] 
		\item\label{IdentityMain} $\lambda(x,x)$ is the identity element of $\gal(E/K)$ for all $x\in Y(E/K,v)$
		\item\label{TransferMain} $\lambda(x,y)(x) = y$ for all $x,y\in Y(E/K,v)$
		\item\label{TransitiveMain} $\lambda(y,z)\lambda(x,y) = \lambda(x,z)$ for all $x,y,z\in Y(E/K,v)$
		\item\label{ContinuousMain} $\lambda$ is continuous.
	\end{enumerate}
\end{thm}

A map $\lambda:Y(E/K,v)\times Y(E/K,v)\to \gal(E/K)$ which satisfies the conditions of Theorem \ref{TransitionExistence} is called a {\it $v$-adic transition diagram on $E$} and we write $T(E/K,v)$ to denote
the set of all such maps.   Because of properties \ref{IdentityMain}, \ref{TransferMain} and \ref{TransitiveMain}, each element of $T(E/K,v)$ creates a commutative diagram 
with objects $E_y$ and morphisms $\lambda(x,y)$. 
As such, we view a $v$-adic transition diagram as a commutative diagram which possesses an additional topological property \ref{ContinuousMain}.



\subsection{Definition of Generalized Ad\`ele} Our formal definition of $\overline{\crazyadele}_E$ requires one additional theorem before proceeding.

\begin{thm} \label{MainTransfer}
	Suppose that $E/F$ is a Galois extension and $K,L\in \I_E$.
	For each place $v$ of $K$, assume that $r_v\in Y(E/K,v)$ and $\lambda_v\in T(E/K,v)$.
	For each place $w$ of $L$, assume that $s_w\in Y(E/L,w)$ and $\mu_w\in T(E/L,w)$.
	If $(a_y)_{y\in Y_E} \in \prod_{y\in Y_E} E_y$ then the following conditions are equivalent:
	\begin{enumerate}[label={(\roman*)}]
		\item $y\mapsto \lambda_v(y,r_v)(a_y)$ is a continuous map $Y(E/K,v)\to E_{r_v}$ for all $v\in Y_K$
		\item $y\mapsto \mu_w(y,s_w)(a_y)$ is a continuous map $Y(E/L,w)\to E_{s_w}$ for all $w\in Y_L$
	\end{enumerate}
\end{thm}

Equipped with Theorem \ref{MainTransfer}, we are able to provide the definition of $\overline{\crazyadele}_E$.

\begin{defin} \label{GenAdeleDef}
	Assume that $E/F$ is a Galois extension and $K\in \I_E$.  For each place $v$ of $K$, let $\lambda_v\in T(E/K,v)$ and $r_v\in Y(E/K,v)$.  
	A point $\mathfrak a = (a_y)_{y\in Y_E} \in \prod_{y\in Y_E} E_y$ is called an {\it ad\`ele of $E$} if the following properties hold:
	\begin{enumerate}[label={(A.\arabic*)}] 
		\item\label{Compact} There exists a compact subset $Z\subseteq Y_E$ such that $a_y \in \mathcal O_y$ for all $y\in Y_E\setminus Z$.
		\item\label{Continuous} The map $y\mapsto \lambda_v(y,r_v)(a_y)$ is a continuous map $Y(E/K,v)\to E_{r_v}$ for all $v\in Y_K$.
	\end{enumerate}
	We shall write $\overline{\crazyadele}_E$ to denote the set of all ad\`eles of $E$.
\end{defin}

Conditions \ref{Compact} and \ref{Continuous} create a rigorous definition inspired by the principles espoused in \ref{BadUnitDisk} and \ref{BadContinuous}.
Furthermore, Theorem \ref{MainTransfer} ensures that Definition \ref{GenAdeleDef} depends only on $E$ and $\mathfrak a$.  Specifically, it is independent of the global field $K$, the $v$-adic transition 
diagrams $\lambda_v$, and the places $r_v$.  Additionally, we are justified in using the term ad\`ele in Definition \ref{GenAdeleDef}.  
Indeed, if $E/F$ is a finite Galois extension then $Y_E$ is discrete, so compactness is equivalent to finiteness and \ref{Continuous} is satisfied for all 
$\mathfrak a$.  In this scenario, we recover the exact definition of ad\`ele for a global field.  

We now wish to equip $\overline{\crazyadele}_E$ with a topology, however, we will need a preliminary theorem.  Still assuming that $E/F$ is a Galois extension and $K\in \I_E$ we let
$v\in Y_K$, $\lambda\in T(E/K,v)$ and $r\in Y(E/K,v)$.  For each set
\begin{equation*}
	U = \prod_{y\in Y_E} U_y \subseteq \prod_{y\in Y_E} E_y
\end{equation*}
we define
\begin{equation*}
	J(K,\lambda,r;U) = \bigcup_{y\in Y(E/K,v)} \left(\{y\}\times \lambda(y,r)(U_y)\right)
\end{equation*}
and note the following important result about $J$.

\begin{thm} \label{TopologyDefine}
	Suppose that $E/F$ is a Galois extension and $K,L\in \I_E$. For each place $v$ of $K$, let $\lambda_v\in T(E/K,v)$ and $r_v\in Y(E/K,v)$.
	For each place $w$ of $L$, let $\mu_w\in T(E/L,w)$ and $s_w\in Y(E/L,w)$.  If
	\begin{equation*}
		U = \prod_{y\in Y_E} U_y \subseteq \prod_{y\in Y_E} E_y
	\end{equation*}
	then the following conditions are equivalent:
	\begin{enumerate}[label={(\roman*)}]
		\item $J(K,\lambda_v,r_v;U)$ is open in $Y(E/K,v)\times E_{r_v}$ for all $v\in Y_K$.
		\item\label{JL} $J(L,\mu_w,s_w;U)$ is open in $Y(E/L,w)\times E_{s_w}$ for all $w\in Y_L$.
	\end{enumerate}
\end{thm}

We are now prepared to equip $\overline{\crazyadele}_E$ with a topology.  Assume that $E/F$ is a Galois extension and $K\in \I_E$, and
for each place $v\in Y_K$, let $\lambda_v\in T(E/K,v)$ and $r_v\in Y(E/K,v)$.  We take as a basis sets of the form $U\cap \overline{\crazyadele}_E$ where
\begin{equation} \label{UForm}
	U = \prod_{y\in Y_E} U_y \subseteq \prod_{y\in Y_E} E_y
\end{equation}
satisfies the following two properties:
\begin{enumerate}[label={(T.\arabic*)}] 
	\item\label{Integers} There exists a compact subset $Z\subseteq Y_E$ such that $U_y = \mathcal O_y$ for all $y\in Y_E\setminus Z$.
	\item\label{Open} $J(K,\lambda_v,r_v;U)$ is open in $Y(E/K,v)\times E_{r_v}$ for all $v\in Y_K$.
\end{enumerate}
As a consequence of Theorem \ref{TopologyDefine}, this definition does not depend on $K$, $\lambda_v$ or $r_v$.  It is straightforward to check that these sets really do form a basis for a topology
on $\overline{\crazyadele}_E$.  When $U$ is of the form \eqref{UForm} satisfying \ref{Integers} and \ref{Open}, we shall often write that $U$ is open in $\overline{\crazyadele}_E$ rather than 
$\overline{\crazyadele}_E\cap U$ is open in $\overline{\crazyadele}_E$.  In the case where $E/F$ is finite, our topology coincides with the usual topology on the classical ad\`ele ring $\overline{\crazyadele}_E$.  
As a result, $\overline{\crazyadele}_E$ is indeed a generalization of the classical ad\`eles to infinite Galois extensions of global fields.

\subsection{Connection to the Direct Limit} We now state our main result which establishes our desired correspondence between $\overline{\directlimit}_E$ and $\overline{\crazyadele}_E$.

\begin{thm}\label{Main}
	If $E/F$ is a Galois extension then the following conditions hold:
	\begin{enumerate}[label={(\roman*)}]
		\item\label{TopologicalRingMain} $\overline{\crazyadele}_E$ is a metrizable topological ring which is complete with respect to any invariant metric on $\overline{\crazyadele}_E$.
		\item\label{CommutativeDiagramMain} For every $K\in \I_E$ there exists a topological subring $\crazyadele_K\subseteq \overline{\crazyadele}_E$ and a topological ring 
			isomorphism $\phi_K:\mathbb A_K\to \crazyadele_K$ such that the diagram
			\begin{center}
			\begin{tikzcd}[row sep=huge, column sep = huge]
				\mathbb A_L \arrow[r,bend left=10,"\phi_L"]& \crazyadele_L \arrow[l,bend left=10,"\phi^{-1}_L"] \\
				 \mathbb A_K \arrow[u,"\con_{L/K}"] \arrow[r,bend left=10,"\phi_K"] &  \crazyadele_K \arrow[l,bend left=10,"\phi^{-1}_K"] \arrow[u,swap,"inclusion"]
			\end{tikzcd}
			\end{center}
			commutes for all $K,L\in \I_E$ with $K\subseteq L$.
		\item\label{ClosureMain} If $\crazyadele_E = \cup_{K\in \I_E} \crazyadele_K$ then $\overline{\crazyadele}_E$ is equal to the closure of $\crazyadele_E$ in $\overline{\crazyadele}_E$.
	\end{enumerate}
\end{thm}

The properties described in Theorem \ref{Main} combine to yield an isomorphism between $\overline{\crazyadele}_E$ and $\overline{\directlimit}_E$.  This isomorphism
accomplishes our goal of recognizing $\overline{\directlimit}_E$ as a space of continuous functions.

\begin{cor} \label{DirectLimitCompletionMain}
	If $E/F$ is a Galois extension then there exists a topological ring isomorphism $\phi: \overline{\directlimit}_E\to\overline{\crazyadele}_E$ such that $\phi(\directlimit_E) = \crazyadele_E$.
\end{cor}

We have already remarked that the metrizability of $\directlimit_E$ can be established directly from the Birkhoff-Kakutani Theorem.  However, we may also conclude this fact as a corollary 
to Theorem \ref{Main}.

\begin{cor} \label{DirectLimitMetrizable}
	If $E/F$ is a Galois extension then $\directlimit_E$ is a metrizable topological ring.
\end{cor}

If $E/F$ is finite then $\directlimit_E$ is equal to the classical ad\`ele ring of $E$.  In this situation, Proposition \ref{ClassicalAdeles} ensures that $\directlimit_E$ is complete with respect to any invariant metric.
By applying Theorem \ref{Main}, we are able to show that $\directlimit_E$ is far from complete in all other cases.

\begin{thm} \label{EqualityEquivalence}
	If $E/F$ is an infinite Galois extension then $\directlimit_E$ has empty interior in $\overline{\directlimit}_E$.
\end{thm}

Because of its relationship with our main goal, we consider $\overline{\crazyadele}_E$ to be a compelling generalization of the ad\`eles to infinite Galois extensions of global fields.  
However, we must caution the reader that it falls short of retaining several important properties of classical ad\`ele rings.  For instance, when $K$ is a global field, 
$\mathbb A_K$ is constructed to be locally compact, a property which establishes the existence of a Haar measure on $\mathbb A_K$.  
Unfortunately, the same property does not necessarily hold for $\overline{\crazyadele}_E$.  Indeed, if there exists $p\in Y_F$ and $y\in Y(E/F,p)$ such that $E_y/F_p$ is an infinite extension, then $E_y$ cannot be locally 
compact, so we have no hope for $\overline{\crazyadele}_E$ to be locally compact.  

It remains open to determine the local compactness of $\overline{\crazyadele}_E$ whenever $E/F$ is an infinite extension such that $E_y/F_p$ is finite for every $p\in Y_F$ and every $y\in Y(E/F,p)$.
Such extensions do indeed exist with certain special cases studied by Checcoli \cite{Checcoli} as well as Checcoli and Zannier \cite{CheccoliZannier}.
Given the significance of local compactness in applications of classical ad\`ele rings, if one wishes to use $\overline{\crazyadele}_E$
in similar ways, surely a resolution to this problem is an important first step.

The subsequent sections of this paper are structured in the following way.  We use Section \ref{ConstructionProofs} to prove Theorems \ref{TransitionExistence}, 
\ref{MainTransfer} and \ref{TopologyDefine}.  Recall that these theorems are the preliminary results which permit our generalized definition of ad\`ele. 
In Sections \ref{TopologicalRingSection}, \ref{CommutativeDiagramSection} and \ref{ClosureSection} we provide our rather extensive proof of Theorem \ref{Main}.  Finally, Section \ref{EqualityProof}
is used to prove Theorem \ref{EqualityEquivalence}.

\section{Proofs Related to the Construction of $\overline{\crazyadele}_E$} \label{ConstructionProofs}

We provide the proof of Theorem \ref{TransitionExistence} by using a preliminary technical lemma.  For this section, we shall write $\mathbb N_0 = \mathbb N\cup\{0\}$.

\begin{lem} \label{BetterNumberFieldStack}
	Suppose $E/F$ is a Galois extension and $K\in \I_E$.   For each $i\in \nat_0$, assume that $K_i\in \I_E$ are such that $K_0 = K$ and $K_i\subseteq K_{i+1}$.
	There exists a map $\lambda:Y(E/K,v)\times Y(E/K,v)\to \gal(E/K)$ satisfying the following properties.
	\begin{enumerate}[label={(\roman*)}]
		\item\label{Identity2} $\lambda(x,x)$ is the identity of $\gal(E/K)$ for all $x\in Y(E/K,v)$
		\item\label{Action2} $\lambda(x,y)(x) = y$ for all $x,y\in Y(E/K,v)$
		\item\label{Transitive2} $\lambda(y,z) \lambda(x,y) = \lambda(x,z)$ for all $x,y,z\in Y(E/K,v)$
		\item\label{Containment2} If $i\in \nat_0$ and $x$ and $y$ are places of $E$ dividing the same place of $K_i$ then $\lambda(x,y) \in \gal(E/K_i)$.
	\end{enumerate}
\end{lem}
\begin{proof}
	We will first prove that for every $n\in \nat_0$ there exists a map $\lambda_n:Y(E/K,v)\times Y(E/K,v)\to \gal(E/K)$ satisfying the following properties:
	\begin{enumerate}[label={(\alph*)}]
		\item\label{Identity} $\lambda_n(x,x)$ is the identity of $\gal(E/K)$ for all $x\in Y(K,v)$
		\item\label{Action} $\lambda_n(x,y)(x) = y$ for all $x,y\in Y(K,v)$
		\item\label{Transitive} $\lambda_n(y,z) \lambda_n(x,y) = \lambda_n(x,z)$ for all $x,y,z\in Y(K,v)$
		\item\label{Containment} If $i\leq n$ and $x$ and $y$ are places of $E$ dividing the same place of $K_i$ then $\lambda_n(x,y) \in \gal(E/K_i)$.
	\end{enumerate}
	We will proceed by induction on $n$ beginning with the base case $n=0$.  Select a place $r\in Y(E/K,v)$.  Since $\gal(E/K)$ acts transitively on $Y(E/K,v)$, for each $x\in Y(E/K,v)$ we may
	select $\sigma_x\in \gal(E/K)$ such that $\sigma_x(x) = r$.  Now define $\lambda_0:Y(E/K,v)\times Y(E/K,v)\to \gal(E/K)$ by
	\begin{equation*}
		\lambda_0(x,y) = \sigma_y^{-1} \sigma_x.
	\end{equation*}
	Under this definition, one easily verifies the four required properties.

	Now assume that $\lambda_n$ is a map satisfying \ref{Identity}, \ref{Action}, \ref{Transitive} and \ref{Containment} and define the following expressions:
	\begin{itemize}
		\item For each $x\in Y(E/K,v)$ we let $w_x$ denote the unique place in $Y(K_{n+1}/K,v)$ such that $x\mid w_x$. 
		\item For each $w\in Y(K_{n+1}/K,v)$, select a place $r_w\in Y(E/K_{n+1},w)$.  In this notation $x\in Y(E/K_{n+1},w_x)$ and $r_{w_x} \in Y(E/K_{n+1},w_x)$.
		\item For each $x\in Y(E/K,v)$ let $\tau_x\in \gal(E/K_{n+1})$ be such that $\tau_x(x) = r_{w_x}$.  Since $\gal(E/K_{n+1})$ acts transitively on $Y(E/K_{n+1},w_x)$, 
			such a map must necessarily exist.
	\end{itemize}
	Now we shall define $\lambda_{n+1}: Y(E/K,v)\times Y(E/K,v)\to \gal(E/K)$ by
	\begin{equation*}
		 \lambda_{n+1}(x,y) = \tau_y^{-1} \lambda_n(r_{w_x},r_{w_y}) \tau_x.
	\end{equation*}
	We must establish properties \ref{Identity}, \ref{Action}, \ref{Transitive} and \ref{Containment} with $n+1$ in place of $n$.  We verify easily that the first three properties hold,
	so it remains to establish \ref{Containment}.  To see this, assume that  $i\leq n+1$ and that $x$ and $y$ are places of $E$ dividing the place $w$ of $K_{i}$.  
	By our assumptions, we know that
	\begin{equation} \label{TauGalois}
		\tau_x,\tau_y\in \gal(E/K_{n+1}) \subseteq \gal(E/K_i).
	\end{equation}
	If $i=n+1$ then $w_x = w_y$ and $\lambda_n(r_{w_x},r_{w_y})$ is equal to the identity by property \ref{Identity}.  This means that $\lambda_{n+1}(x,y) = \tau_y^{-1} \tau_x$
	and it follows that $\lambda_{n+1}(x,y) \in \gal(E/K_i)$ as required.  

	Now assuming that $i\leq n$, we know that $x$ and $r_{w_x}$ both divide the same place of $K_{n+1}$, and hence, they must divide the same place of $K_i$.  That is, 
	$x$ and $r_{w_x}$ both divide $w$.  Similarly, $y$ and $r_{w_y}$ both divide $w$ as well.  We obtain from \ref{Containment} that $\lambda_n(r_{w_x},r_{w_y}) \in \gal(E/K_i)$.
	Now applying the definition of $\lambda_{n+1}$ along with \eqref{TauGalois}, we conclude that $\lambda_{n+1}(x,y)\in \gal(E/K_i)$.  
	
	We may now assume that, for every $n\in \nat_0$, $\lambda_n:Y(E/K,v)\times Y(E/K,v)\to \gal(E/K)$ satisfies \ref{Identity}, \ref{Action}, \ref{Transitive} and \ref{Containment}.
	Let $G = \gal(E/K)$ and $I = Y(E/K,v)\times Y(E/K,v)$.  Write $G^I$ to denote the set of functions from $I$ to $G$ and equip $G^I$ with the product topology.
	If $\{\lambda_n\}_{n=0}^\infty$ is a finite list of distinct points in $G^I$, then this sequence has a constant subsequence and we may take $\lambda$ to be that constant.
	In this case, we immediately obtain the required properties.
	
	We now assume this sequence defines an infinite set.  According to Tychonoff's Theorem \cite[Ch. 5, Theorem 37.3]{Munkres}, $G^I$ is compact, and therefore, $G^I$ is limit point compact.  
	As a result, $\{\lambda_n: n\in \nat_0\}$ must have a limit point $\lambda\in G^I$, so in particular, $\lambda(x,y)$ is a limit point of $\{\lambda_n(x,y): n\in \nat_0\}$ for all $x,y\in Y(E/K,v)$.  
	We now verify that $\lambda$ satisfies the four required properties.
	
	For \ref{Identity2}, assume that $\lambda(x,x) \ne 1$.  Since $G$ is Hausdorff, there exists an open neighborhood $U$ of $\lambda(x,x)$ such that $1\not\in U$.  We may select
	$n\in \nat_0$ such that $\lambda_n(x,x)\in U$, so \ref{Identity} implies that $1\in U$, a contradiction.
		
	To prove \ref{Action2}, suppose that $\lambda(x,y)(x)\ne y$ and select an open neighborhood $U$ of $\lambda(x,y)(x)$ such that $y\not\in U$.  Now define $f:G\times Y(E/K,v)\to Y(E/K,v)$
	by $f(\sigma,x) = \sigma(x)$ and observe that $(\lambda(x,y),x)\in f^{-1}(U)$.  However, \cite[Lemma 3]{AllcockVaaler} asserts that $f$ is continuous, so there exist open sets
	$A \subseteq G$ and $B \subseteq Y(E/K,v)$ such that
	\begin{equation*}
		(\lambda(x,y),x) \in A\times B \subseteq f^{-1}(U).
	\end{equation*}
	Now we may let $n\in \nat_0$ be such that $\lambda_n(x,y)\in A$.  We obtain that $(\lambda_n(x,y),x) \in A\times B \subseteq f^{-1}(U)$ which means that
	$\lambda_n(x,y)(x) \in U$.  Then using \ref{Action}, we get $y\in U$, a contradiction.
			
	In order to establish \ref{Transitive2}, assume that $\lambda(y,z)\lambda(x,y) \ne \lambda(x,z)$ so that $$\lambda(y,z)\lambda(x,y)\lambda(x,z)^{-1} \ne 1$$ and select an open neighborhood $U$ of 
	$\lambda(y,z)\lambda(x,y)\lambda(x,z)^{-1}$ such that $1\not\in U$.  We define $g:G\times G\times G \to G$ by $g(\sigma, \tau,\rho) = \sigma\tau\rho^{-1}$ and note that
	$$(\lambda(y,z),\lambda(x,y),\lambda(x,z)) \in g^{-1}(U).$$  Since $G$ is a topological group, $g$ must be continuous, so there exist open sets $U_1, U_2, U_3\subseteq G$ such that
	\begin{equation*}
		(\lambda(y,z),\lambda(x,y),\lambda(x,z)) \in U_1\times U_2\times U_3 \subseteq g^{-1}(U).
	\end{equation*}
	However, $\lambda$ is a limit point of $\lambda_n$ in $G^I$ with the product topology, so there exists $n\in \nat_0$ such that
	\begin{equation*}
		(\lambda_n(y,z),\lambda_n(x,y),\lambda_n(x,z)) \in U_1\times U_2\times U_3 \subseteq g^{-1}(U).
	\end{equation*}
	It follows that $\lambda_n(y,z)\lambda_n(x,y)\lambda_n(x,z)^{-1} \in U$, and then \ref{Transitive} yields $1\in U$, a contradiction.
	
	Finally, to show that \ref{Containment2} holds, suppose that $i\in \nat_0$ and that $x$ and $y$ divide the same place of $K_i$.  Further assume that $\lambda(x,y)\not\in \gal(E/K_i)$ and note that 
	$\gal(E/K_i)$ is closed in $G$.  Now define
	\begin{equation*}
		B = \gal(E/K_i) \cup\{\lambda_n(x,y): n < i \mbox{ and }\ \lambda_n(x,y)\ne \lambda(x,y)\}
	\end{equation*}
	so that $B$ is closed.  Clearly $\lambda(x,y)\not \in B$ so we select an open neighborhood $U$ of $\lambda(x,y)$
	which contains no points from $B$.  There exists $n\in \nat_0$ such that $\lambda_n(x,y)\in U\setminus\{\lambda(x,y)\}$.  We cannot have $n < i$ because then $\lambda_n(x,y)\in B$,
	which would contradict our assumptions about $U$.  This means that $n\geq i$ and \ref{Containment} implies that $\lambda_n(x,y)\in \gal(E/K_i)$, also a contradiction.
\end{proof}

\begin{proof}[Proof of Theorem \ref{TransitionExistence}]
	We know that $\I_E$ is countable, so we may assume that $\{K'_i: i \in \nat_0\} = \I_E$ is such that $K = K'_0$.  
	We now recursively define $K_i$ by setting $K_0 = K'_0$ and $K_{i} = K_{i-1}K'_{i}$ for all $i\in \nat$.  From this definition, we obtain that the $K_i$ satisfy the assumptions of Lemma 
	\ref{BetterNumberFieldStack}.  Therefore, there exists a map $\lambda:Y(E/K,v)\times Y(E/K,v)\to \gal(E/K)$ satisfying properties \ref{Identity2},\ref{Action2}, \ref{Transitive2} and \ref{Containment2}.
	As a result, $\lambda$ immediately is known to satisfy the first three required properties.
	
	To establish continuity, we shall assume that $H$ is an open normal subgroup of $G := \gal(E/K)$ and $\sigma \in G$.  We must prove that $\lambda^{-1}(\sigma H)$ is open in $Y(E/K,v)\times Y(E/K,v)$.
	There exists a finite Galois extension $L/K$ with $L\subseteq E$ such that $H = \gal(E/L)$.  The Galois closure $L'$ of $L$ over $F$ belongs to $\mathcal I_E$, and hence, there exists $i\in \nat_0$ such
	$L' = K'_i$.  This means that $L\subseteq K_i$, and as a result, Lemma \ref{BetterNumberFieldStack}\ref{Containment2} implies that 
	\begin{equation} \label{DiagInterior}
		Y(E/K_i,w)\times Y(E/K_i,w)\subseteq \lambda^{-1}( \gal(E/K_i)) \subseteq \lambda^{-1}(\gal(E/L)) = \lambda^{-1}(H)
	\end{equation}
	for all places $w$ of $K_i$ that divide $v$.
	
	Now let $(x,y)\in \lambda^{-1}(\sigma H)$ so that
	\begin{equation} \label{LambdaContain}
		\sigma^{-1}\lambda(x,y) \in H.
	\end{equation}
	It follows from \eqref{DiagInterior} that there exist open sets $U$ and $V$ in $Y(E/K,v)$ such that
	\begin{equation} \label{XContain}
		(x,x)\subseteq U\times U \subseteq \lambda^{-1}(H)\quad\mbox{and}\quad (y,y)\subseteq V\times V\subseteq \lambda^{-1}(H).
	\end{equation}
	To finish proving our claim that $\lambda$ is continuous, it is enough to show that 
	\begin{equation*} \label{XYContain}
		(x,y) \in U\times V\subseteq \lambda^{-1}(\sigma H).
	\end{equation*}
	It is obvious that $(x,y)\in U\times V$ in view of \eqref{XContain}.  Therefore, we assume that $(x_0,y_0)\in U\times V$ and note that Lemma \ref{BetterNumberFieldStack} \ref{Transitive2} yields
	\begin{equation} \label{Transitivity}
		\sigma^{-1}\lambda(x_0,y_0) = \sigma^{-1}\lambda(y,y_0)\lambda(x,y)\lambda(x_0,x).
	\end{equation}
	Since $(y,y_0)\in V\times V$, we obtain from \eqref{XContain} that $\lambda(y,y_0)\in H$, and since $H$ is normal in $G$, we also conclude that $\sigma^{-1}\lambda(y,y_0)\in H\sigma^{-1}$.
	In other words, there exists $h\in H$ such that $\sigma^{-1}\lambda(y,y_0) = h\sigma^{-1}$.  Then we apply \eqref{Transitivity} to deduce that
	\begin{equation*}
		\sigma^{-1}\lambda(x_0,y_0) = h\sigma^{-1} \lambda(x,y)\lambda(x_0,x).
	\end{equation*}
	We already know that $h\in H$, from \eqref{LambdaContain} we have that $\sigma^{-1} \lambda(x,y) \in H$, and finally, the first statement of \eqref{XContain} shows that $\lambda(x_0,x)\in H$.
	It follows that $\sigma^{-1}\lambda(x_0,y_0) \in H$ so that $(x_0,y_0)\in\lambda^{-1}(\sigma H)$, as required.
		
\end{proof}

The proof of Theorem \ref{MainTransfer} requires a lemma establishing the continuity of a certain map.  If $y$ is a place of $E$, we shall write $B_y(c,\rho) = \{\alpha\in E_y:|\alpha - c|_y < \rho\}$
for the ball centered at $c\in E_y$ of radius $\rho > 0$.

\begin{lem} \label{SContinuity}
	Suppose $E/F$ is a Galois extension and $p\in Y_F$.  Let $r,s\in Y(E/F,p)$ and define $S = \{\sigma\in \gal(E/F): \sigma(r) = s\}$.
	Then the map $\psi:S\times E_r\to E_s$ given by $\psi(\sigma,\alpha)= \sigma(\alpha)$ is continuous.
\end{lem}
\begin{proof}	
	Assume $U$ is an open subset of $E_s$ and let $(\sigma,\alpha)\in \psi^{-1}(U)$.  Therefore, $\sigma(\alpha)\in U$ and there exists $\varepsilon > 0$ such that
	\begin{equation} \label{UBall}
		B_s(\sigma(\alpha),\varepsilon) \subseteq U.
	\end{equation}
	Since $E$ is dense in $E_s$ we may choose $\beta\in E$ such that 
	\begin{equation} \label{BetaApprox}
		|\sigma(\beta) - \sigma(\alpha)|_s < \frac{\varepsilon}{2}.
	\end{equation}
	Now let $L\in I_E$ be such that $\beta\in L$ and define $H = \gal(E/L)$ so that $H$ is an open subgroup of $\gal(E/F)$.
	It follows that $\sigma H$ is open in $\gal(E/F)$ and $\sigma H \cap S$ is open in $S$.  We define
	\begin{equation*}
		V = (\sigma H \cap S) \times B_r\left(\beta,\frac{\varepsilon}{2}\right)
	\end{equation*}
	and note that $V$ is open in $S\times E_r$.  Clearly $\sigma\in \sigma H\cap S$.  Also
	\begin{equation*}
		|\beta - \alpha|_r = |\sigma(\beta) - \sigma(\alpha)|_{\sigma(r)} = |\sigma(\beta) - \sigma(\alpha)|_s < \frac{\varepsilon}{2}
	\end{equation*}
	which means that $\alpha\in B_r(\beta,\varepsilon/2)$.  Therefore, $(\sigma,\alpha)\in V$ and it remains only to show that $V\subseteq \phi^{-1}(U)$.
	
	Assume that $(\tau,\gamma)\in V$ so that $\gamma\in \sigma H\cap S$ and $|\gamma - \beta|_r < \varepsilon/2$.  We get from \eqref{BetaApprox} that
	\begin{equation*}
		|\tau(\gamma) - \sigma(\alpha)|_s = |\tau(\gamma) - \sigma(\beta) + \sigma(\beta) -\sigma(\alpha)|_s < |\tau(\gamma) - \sigma(\beta)|_s + \frac{\varepsilon}{2}.
	\end{equation*}
	We may select $h\in H$ such that $\tau = \sigma h$.  This means that $h(\beta) = \beta$ and
	\begin{equation*}
		|\tau(\gamma) - \sigma(\beta)|_s = |\tau(\gamma) - \sigma(h(\beta))|_s = |\tau(\gamma) - \tau(\beta)|_s = |\gamma - \beta|_r < \frac{\varepsilon}{2}.
	\end{equation*}
	Combining these observations yields that $|\tau(\gamma) - \sigma(\alpha)|_s < \varepsilon$ and it follows from \eqref{UBall} that $\tau(\gamma)\in U$.
\end{proof}

\begin{proof}[Proof of Theorem \ref{MainTransfer}]
	Let $f_v:Y(E/K,v)\to E_{r_v}$ and $g_w:Y(E/L,w)\to E_{s_w}$ be given by
	\begin{equation*}
		f_v(y) = \lambda_v(y,r_v)(a_y)\quad\mbox{and}\quad g_w(y) = \mu_w(y,s_w)(a_y).
	\end{equation*}
	Without loss of generality, it is enough to assume that $f_v$ is continuous for all $v\in Y_K$ and to show that $g_w$ is continuous for all $w\in Y_L$.  
	Thus, we assume that $f_v$ is continuous for all $v\in Y_K$.
	
	Assume that $t\in Y(E/L,w)$ and let $v\in Y_K$ be such that $t\in Y(E/K,v)$.  Set $T = Y(E/K,v)\cap Y(E/L,w)$ so that $T$ is an open neighborhood of $t$.  To complete the proof, it is enough
	to show that $g_w$ is continuous on $T$.  To see this, we let $S = \{\sigma\in \gal(E/F): \sigma(r_v) = s_w\}$ and define the following four maps:
	\begin{align*}
		 d: T\to T\times T \quad\quad & d(x) = d(x,x) \\
		 f:T\to E_{r_v} \quad\quad & f(x) = \lambda_v(y,r_v)(a_y) \\
		 \phi: T\to S\quad\quad & \phi(x) = \mu_w(y,s_w)\lambda_v(r_v,y) \\
		 \psi: S \times E_{r_v}\to E_{s_w}\quad\quad & \psi(\sigma,\alpha) = \sigma(\alpha).
	\end{align*}
	For every point $y\in T$, we have that
	\begin{equation*}
		g_w(y) = (\mu_w(y,s_w)\lambda_v(r_v,y)\lambda_v(y,r_v))(a_y) = \psi(\phi(y),f(y)) = (\psi\circ(\phi\times f)\circ d)(y),
	\end{equation*}
	so to complete the proof, it is sufficient to establish that $d, f, \phi$ and $\psi$ are continuous.
	
	Clearly $d$ is continuous, and furthermore, $f$ is simply the restriction of $f_v$ to $T$, so it is also continuous.  The continuity of $\phi$ follows from property \ref{ContinuousMain} and the fact
	that $\gal(E/F)$ is a topological group.  Finally, Lemma \ref{SContinuity} shows that $\psi$ is continuous establishing the theorem.

\end{proof}

\begin{proof}[Proof of Theorem \ref{TopologyDefine}]
	Assume that $J(K,\lambda_v,r_v;U)$ is open in $Y(E/K,v)\times E_{r_v}$ for all $v\in Y_K$ and let $M = KL$.  Now fix a place $w\in Y_L$ and consider
	\begin{equation*}
		J(L,\mu_w,s_w;U) = \bigcup_{z\in Y(M/L,w)}\left( \bigcup_{y\in Y(E/M,z)} \left( \{y\}\times \mu_w(y,s_w)(U_y)\right)\right).
	\end{equation*}
	To establish \ref{JL}, we shall assume that $z\in Y(M/L,w)$ and let
	\begin{equation} \label{SpecialJ}
		A = \bigcup_{y\in Y(E/M,z)} \left( \{y\}\times \mu_w(y,s_w)(U_y)\right).
	\end{equation}
	It is certainly enough to prove that $A$ is open in $Y(E/L,w)\times E_{s_w}$.  However, since $Y(E/M,z)\times E_{s_w}$ is open in $Y(E/L,w)\times E_{s_w}$, it is actually sufficient to prove that $A$
	is open in $Y(E/M,z)\times E_{s_w}$.
	
	To see this, assume that $v$ is the unique place of $K$ such that $z\mid v$.  We have assumed that $J(K,\lambda_v,r_v;U)$ is open in $Y(E/K,v)\times E_{r_v}$, and therefore, if we set 
	$B = J(K,\lambda_v,r_v;U) \cap (Y(E/M,z)\times E_{s_w})$, we conclude that $B$ is open in $Y(E/M,z)\times E_{s_w}$.  However,
	\begin{align*}
		B & = \left(\bigcup_{y\in Y(E/K,v)} \left(\{y\}\times \lambda_v(y,r_v)(U_y)\right)\right) \cap \left(Y(E/M,z)\times E_{r_v}\right) \\
		& = \left(\bigcup_{y\in Y(E/M,z)} \left(\{y\}\times \lambda_v(y,r_v)(U_y)\right)\right) \cap \left(Y(E/M,z)\times E_{r_v}\right) \\
		& = \bigcup_{y\in Y(E/M,z)} \left( \{y\}\times  \lambda_v(y,r_v)(U_y)\right).
	\end{align*}
	Now we define the map $f:Y(E/M,z)\times E_{s_w}\to Y(E/M,z)\times E_{r_v}$ by 
	\begin{equation*}
		f(y,\alpha) =  (y, \lambda_v(y,r_v)\mu_w(s_w,y)(\alpha)).
	\end{equation*}
	By applying Lemma \ref{SContinuity}, we may verify that $f$ is a continuous bijection, and moreover, it satisfies
	\begin{equation*}
		f(A) = \bigcup_{y\in Y(E/M,z)} f(\left( \{y\}\times \mu_w(y,s_w)(U_y)\right)) =  \bigcup_{y\in Y(E/M,z)} \left( \{y\}\times \lambda_v(y,r_v)(U_y)\right) = B,
	\end{equation*}
	where the penultimate equality follows from \ref{IdentityMain} and \ref{TransitiveMain}.  Therefore, $A = f^{-1}(B)$ and the result follows.
\end{proof}

\section{Proof of Theorem \ref{Main}\ref{TopologicalRingMain}} \label{TopologicalRingSection}

Now that we have defined $\overline{\crazyadele}_E$ and equipped it with a topology, we may proceed with our proof of Theorem \ref{Main}\ref{TopologicalRingMain} which begins with a general topological lemma.  
For Lemmas \ref{OpenTube} and \ref{FieldTopRing}, we shall write $B(x,r)$ to denote the open ball in $X$ centered at $x$ of radius $r$.

\begin{lem} \label{OpenTube}
	Suppose $I$ is a compact space and $X$ is a field with an absolute value $|\ |$.  Further assume that $f:I\to X$ is continuous and $\Gamma_i\subseteq X$ is such that $f(i)\in \Gamma_i$ for all $i\in I$.  
	Finally, assume that
	\begin{equation*}
		\bigcup_{i\in I} \left(\{i\}\times \Gamma_i\right)
	\end{equation*}
	is open in $I\times X$.  Then there exists $\varepsilon > 0$ such that $B(f(i),\varepsilon) \subseteq \Gamma_i$ for all $i\in I$.
\end{lem}
\begin{proof}
	Define a function $g\colon I\times X\to I\times X$ by $g(i,x)=(i,x-f(i))$. Clearly $g$ is invertible with $g^{-1}(i,x)=(i,x+f(i))$, and it is also clear that $g$ and $g^{-1}$ are continuous making a $g$ a homeomorphism. 
	For each $i\in I$, let $\widetilde{\Gamma}_i=\{x\in X\colon x+f(i)\in\Gamma_i\}$. Then 
	\begin{equation*}
		I\times\{0\}\subseteq\bigcup_{i\in I}\left(\{i\}\times\widetilde{\Gamma}_i\right)=g\left(\bigcup_{i\in I}\left(\{i\}\times\Gamma_i\right)\right)
	\end{equation*}
	which is open in $I\times X$ since $g$ is a homeomorphism.  If $c\in X$ we shall write $B(c,\varepsilon) = \{x\in X: |x - c| < \varepsilon\}$.
	Since $I$ is compact, by the Tube Lemma \cite[Ch.~3, Lemma~26.8]{Munkres} there exists $\varepsilon>0$ such that 
	\begin{equation*}
		I\times B(0,\varepsilon)\subseteq\bigcup_{i\in I}\left(\{i\}\times\widetilde{\Gamma}_i\right).
	\end{equation*}
	It follows that for each $i\in I$ we have $B(0,\varepsilon)\subseteq\widetilde{\Gamma}_i$, and thus, $B(f(i),\varepsilon)\subseteq\Gamma_i$.
\end{proof}

If $X$ is any ring we shall define the maps $\Add:X\times X\to X$ and $\Mult:X\times X\to X$ is the obvious ways:
\begin{equation*}
	\Add(x,y) = x+y\quad\mbox{and}\quad \Mult(x,y) = xy.
\end{equation*}
The following basic lemma describes the behavior of these maps on an arbitrary field with absolute value.

\begin{lem}\label{FieldTopRing}
	Suppose that $X$ is a field with absolute value $|\ |$.  If $x,y\in X$ and $r>0$ then
	\begin{equation*}
		B\left( x,\frac{r}{2}\right)\times B\left( y,\frac{r}{2}\right) \subseteq \Add^{-1}\left( B(x+y,r)\right)
	\end{equation*}
	and
	\begin{equation*}
		B\left( x,\min\left\{1,\frac{r}{1 + |x| + |y|}\right\}\right)\times B\left( y,\min\left\{1,\frac{r}{1 + |x| + |y|}\right\}\right)\subseteq \Mult^{-1}\left( B(xy,r)\right)
	\end{equation*}
\end{lem}
\begin{proof}
	If $(a,b)\in B(x,r/2)\times B(y,r/2)$ then
	\begin{equation*}
		|a+b - (x+y)| \leq |a-x| + |b-y| < \frac{r}{2} + \frac{r}{2} = r
	\end{equation*}
	and the first statement follows.  Now assume that $$(a,b)\in B\left(x,\min\left\{1,\frac{r}{1+|x| + |y|}\right\}\right)\times B\left(y,\min\left\{1,\frac{r}{1+|x| + |y|}\right\}\right).$$  If $1 \leq r/(1+|x| + |y|)$ then
	\begin{align*}
		|ab- xy| & = |ab - ay + ay - xy| \\
			& \leq |a|\cdot |b-y| + |y|\cdot|a-x| \\
			& \leq |a|+|y| \\
			& \leq |a-x| + |x| + |y| \\
			& <1 + |x| + |y| \\
			& \leq r.
	\end{align*}
	On the other hand, if $1 > r/(1+|x| + |y|)$ then we observe that
	\begin{equation*}
		|ab - xy| \leq  \frac{r(|a| + |y|)}{1 + |x| + |y|} \leq \frac{r(|a-x| + |x| + |y|)}{1 + |x| + |y|} < r
	\end{equation*}
	and the result follows.
\end{proof}

Before proceeding with the proof of Theorem \ref{Main}\ref{TopologicalRingMain}, we simplify our use of Cauchy, complete and completion.
Suppose that $R$ is a metrizable topological ring and $d$ is an invariant metric on $R$.  Given a sequence $\{\alpha_n\}_{n=1}^\infty$ in $R$, 
we recall that the following conditions are equivalent:
\begin{enumerate}[label={(\roman*)}]
	\item $\{\alpha_n\}_{n=1}^\infty$ is Cauchy with respect to $d$
	\item\label{NbhdCauchy2} For every open neighborhood $U$ of $0$ there exists $N\in \nat$ such that $\alpha_m - \alpha_n \in U$ for all $m,n\geq N$.
\end{enumerate}
We say that $\{\alpha_n\}_{n=1}^\infty$ is {\it Cauchy in $R$} if it is Cauchy with respect to $d$.  Similarly, $R$ is {\it complete} if it is complete with respect to $d$, 
and the {\it completion} $\overline R$ of $R$ is the completion of $G$ with respect to $d$.  In view of \ref{NbhdCauchy2}, these definitions depend only on the algebraic
and topological properties of $R$, so they are independent of the choice of invariant metric $d$.  As such, a sequence is Cauchy if and only if it is Cauchy with respect to every invariant metric on $R$.
Furthermore, our definition of $\overline R$ agrees with the definition used in the introduction.

\begin{proof}[Proof of Theorem \ref{Main}\ref{TopologicalRingMain}]
	Using the fact that $E_y$ is a topological field, it is straightforward to verify the ring axioms of $\overline{\crazyadele}_E$.  Now fix $K\in \I_E$, and for each place
	$v$ of $K$, we let $\lambda_v\in T(E/K,v)$ and $r_v\in Y(E/K,v)$. It remains to prove that the ring operations define continuous maps $\overline{\crazyadele}_E\times \overline{\crazyadele}_E \to \overline{\crazyadele}_E$, that
	$\overline{\crazyadele}_E$ is a metric space, and that $\overline{\crazyadele}_E$ is complete.
	
	{\bf Proof that the ring operations are continuous:} We begin with addition.  We shall assume that
	\begin{equation*}
		U = \prod_{y\in Y_E} U_y \subseteq \prod_{y\in Y_E} E_y
	\end{equation*}
	is a basis element in $\overline{\crazyadele}_E$, i.e., $U$ satisfies properties \ref{Integers} and \ref{Open}.  According to \ref{Open}, we know that 
	\begin{equation} \label{JRep}
		J(K,\lambda_v,r_v;U) = \bigcup_{y\in Y(E/K,v)}\left( \{y\}\times \lambda_v(y,r_v)(U_y)\right)
	\end{equation}
	is open in $Y(E/K,v)\times E_{r_v}$ for all places $v$ of $K$.  Suppose that $\mathfrak a = (a_y)_{y\in Y_E}$ and $\mathfrak b = (b_y)_{y\in Y_E}$ are points in $\overline{\crazyadele}_E$ such 
	that  $(\mathfrak a,\mathfrak b)\in \Add^{-1}(U)$.  
	
	By combining \ref{Compact} and \ref{Integers}, we obtain a compact set $Z\subseteq Y_E$ such that $a_y,b_y\in \mathcal O_y$ and $U_y = \mathcal O_y$ for all $y\in Y_E\setminus Z$.
	Therefore, we may assume that $S$ is a finite set of places of $K$ containing all archimedean places such that $Z \subseteq \cup_{v\in S} Y(E/K,v)$.
	Certainly $a_y,b_y\in \mathcal O_y$ and $U_y = \mathcal O_y$ for all $y$ not dividing a place in $S$.
	
	For each place $v\in Y_K$, define $f_v,g_v:Y(E/K,v)\to E_r$ by $f_v(y) = \lambda_v(y,r_v)(a_y)$ and $g(y) = \lambda_v(y,r_v)(b_y)$.  Since $\mathfrak a,\mathfrak b \in \overline{\crazyadele}_E$, 
	$f_v$ and $g_v$ are both continuous, and since $E_{r_v}$ is a topological field, $f_v +g_v$ is also continuous.   Moreover, $f_v(y) + g_v(y)\in \lambda_v(y,r_v)(U_y)$ for all $y\in Y(E/K,v)$.
	Then by applying Lemma \ref{OpenTube} and \eqref{JRep}, there exists $\varepsilon_v >0$ such that $B_r(f_v(y)+g_v(y),\varepsilon_v) \subseteq \lambda_v(y,r_v)(U_y)$ for all $y\in Y(E/K,v)$.  
	Using the fact that $\lambda_v(y,r_v)$ is an isometric isomorphism, we conclude that
	\begin{equation} \label{fgContain}
		B_y(a_y + b_y,\varepsilon_v) \subseteq U_y\quad\mbox{for all } y\in Y(E/K,v).
	\end{equation}
	Now set
	\begin{equation*}
		G_y = \begin{cases} B_y(a_y,\varepsilon_v/2) & \mbox{if } y \mbox { divides a place } v\in S \\ \mathcal O_y & \mbox{if } y \mbox{ does not divide a place in } S\end{cases}
	\end{equation*}
	and
	\begin{equation*}
		H_y = \begin{cases} B_y(b_y,\varepsilon_v/2) & \mbox{if } y \mbox { divides a place } v\in S \\ \mathcal O_y & \mbox{if } y \mbox{ does not divide a place in } S.\end{cases}
	\end{equation*}
	We define
	\begin{equation*}
		V = \prod_{y\in Y_E}G_y\quad\mbox{and}\quad W =  \prod_{y\in Y_E} H_y.
	\end{equation*}
	We must now show that $V$ and $W$ are open in $\overline{\crazyadele}_E$ and that $(\mathfrak a,\mathfrak b)\in V\times W \subseteq \Add^{-1}(U)$.  Certainly $(\mathfrak a,\mathfrak b)\in V \times W$
	and $V$ and $W$ satisfy \ref{Integers}.  Additionally, the continuity of $f_v$ and $g_v$ ensures that $V$ and $W$ satisfy \ref{Open} and we conclude that $V$ and $W$ are open.
	
	It remains to show that $V\times W\subseteq \Add^{-1}(U)$.  To see this, we suppose that $\mathfrak c = (c_y)_{y\in Y_E}$ and $\mathfrak d = (d_y)_{y\in Y_E}$ are such that 
	$(\mathfrak c,\mathfrak d)\in V\times W$.  If $y$ does not divide a place in $S$, then $c_y,d_y\in \mathcal O_y$, and hence, $c_y + d_y\in \mathcal O_y$.  But $U_y = \mathcal O_y$ for
	all $y$ not dividing a place in $S$, and therefore, $c_y + d_y\in U_y$.  If $y$ divides a place $v\in S$ then $(c_y,d_y)\in B_y(a_y,\varepsilon_v/2) \times B_y(b_y,\varepsilon_v/2)$.
	According to Lemma \ref{FieldTopRing}, we obtain that $c_y + d_y \in B_y(a_y + b_y,\varepsilon_v)$ and it follows from \eqref{fgContain} that $c_y+d_y\in U_y$.  We have now established
	that $\Add:\overline{\crazyadele}_E\times \overline{\crazyadele}_E\to \overline{\crazyadele}_E$ defines a continuous map.
	
	Next, suppose that $(\mathfrak a,\mathfrak b)\in \Mult^{-1}(U)$ and select $S$ in the same way as before.  In this case,
	we apply Lemma \ref{OpenTube} to obtain $\varepsilon_v >0$ such that $B_v(f_v(y)g_v(y),\varepsilon_v) \subseteq \lambda_v(y,r_v)(U_y)$ for all $y\in Y(E/K,v)$.
	Since $f_v$ and $g_v$ are continuous and $Y(E/K,v)$ is compact we may define $m_v = \max\{1 + |f_v(y)|_r + |g_v(y)|_r:y\in Y(E/K,v)\}$.  The we set
	\begin{equation*}
		G_y = \begin{cases} B_y(a_y,\min\{1,\varepsilon_v/m_v\}) & \mbox{if } y \mbox { divides a place } v\in S \\ \mathcal O_y & \mbox{if } y \mbox{ does not divide a place in } S\end{cases}
	\end{equation*}
	and
	\begin{equation*}
		H_y = \begin{cases} B_y(b_y,\min\{1,\varepsilon_v/m_v\}) & \mbox{if } y \mbox { divides a place } v\in S \\ \mathcal O_y & \mbox{if } y \mbox{ does not divide a place in } S.\end{cases}
	\end{equation*}
	The define
	\begin{equation*}
		V = \prod_{y\in Y_E}G_y\quad\mbox{and}\quad W =  \prod_{y\in Y_E} H_y.
	\end{equation*}
	From a similar application of Lemma \ref{FieldTopRing} we have completed the proof that $\overline{\crazyadele}_E$ is a topological ring.  
	
	{\bf Proof that $\overline{\crazyadele}_E$ is a metric space:}  According to the Birkhoff-Kakutani Theorem (see \cite{Birkhoff,Kakutani}), it is sufficient to show that $\{0\}$ is closed and that $0$ has a 
	countable base of neighborhoods.  If $\mathfrak a = (a_y)_{y\in Y_E} \ne 0$ then assume $x\in Y_E$ is such that $a_x\ne 0$.  We may assume that $v$ is the unique place of $K$ with $x\mid v$.
	Furthermore, using \ref{Compact}, we may assume that $S$ is a finite set of places of $K$ containing $v$ and all archimedean places such that $a_y \in \mathcal O_y$ for all 
	$y$ not dividing a place in $S$.  If we set
	\begin{equation*}
		U_y = \begin{cases} B_y(a_y,|a_x|_x) & \mbox{if } y \mbox{ divides a place in } S \\
			\mathcal O_y & \mbox{if } y \mbox{ does not divide a place in } S.
			\end{cases}
	\end{equation*}
	then using \ref{Continuous} we find that $U = \prod_{y\in Y_E} U_y$ is an open neighborhood of $\mathfrak a$ not containing $0$.
	
	We now show that $0$ has a countable base of neighborhoods.  For this purpose, let $Y_{K,\infty}$ denote the set of archimedean places of $K$ so that $Y_{K,\infty}$ is finite.
	Also define
	\begin{equation*}
		\Omega = \{S\subseteq Y_K: \#S < \infty\mbox{ and } Y_{K,\infty} \subseteq S\}
	\end{equation*}
	and let $\mathscr B$ be the set of all open neighborhoods of $0$ in $\overline{\crazyadele}_E$.
	Since $\Omega\times \nat$ is countable, it is sufficient to identify a map $\phi:\Omega\times \nat\to \mathscr B$ such that $\phi(\Omega\times \nat)$ is a base of neighborhoods of $0$.
	For each $(S,\alpha)\in \Omega \times \nat$ and $y\in Y_E$ set
	\begin{equation*}
		\phi_y(S,n) = \begin{cases} B_y(0,1/n) & \mbox{if } y \mbox{ divides a place in } S \\
								\mathcal O_y &\mbox{if } y \mbox{ does not divide a place in } S
					\end{cases}
	\end{equation*}
	and define
	\begin{equation*}
		\phi(S,n) = \prod_{y\in Y_E} \phi_y(S,n).
	\end{equation*}
	We know that that $\phi(S,n)$ is a well-defined map to $\mathscr B$.   
	
	Now assume that $U = \prod_{y\in Y_E} U_y\in \mathscr B$ so there exists a compact set $Z\subseteq Y_E$ such that $U_y = \mathcal O_y$ for all $y\in Y_E\setminus Z$.  
	In particular, there exists $S\in \Omega$ such that $Z\subseteq \bigcup_{v\in S} Y(E/K,v)$.  For each $v\in S$, Lemma \ref{OpenTube} ensures that
	there exists $\varepsilon_v >0$ such that $B_{r_v}(0,\varepsilon_v)\subseteq \lambda_v(y,r_v)(U_y)$ for all $y\in Y(E/K,v)$.  If we let $\varepsilon =\min_{v\in S}\{\varepsilon_v\}$ then certainly 
	$B_{r_v}(0,\varepsilon)\subseteq \lambda_v(y,r_v)( U_y)$ for all $y$ dividing a place in $v\in S$.  Now choose $n\in \nat$ such that $1/n \leq \varepsilon$ and it follows that 
	$B_{r_v}(0,1/n) \subseteq \lambda_v(y,r_v)(U_y)$ for all $y$ dividing a place in $v\in S$.  $\lambda_v(y,r_v)$ is an isometric isomorphism so we get $B_y(0,1/n)\subseteq U_y$ for 
	all $y$ dividing a place in $S$, and hence $\phi(S,n) \subseteq U$.
	
	{\bf Proof that $\overline{\crazyadele}_E$ is complete:} We assume that $\mathfrak a_n = (a_{n,y})_{y\in Y_E}$ is a Cauchy sequence in $\overline{\crazyadele}_E$.
	By setting $B = \prod_{y\in Y_E} \left\{ \alpha\in E_y:|\alpha|_y \leq 1\right\}$, there must exist $N_0\in \nat$ such that 
	$\mathfrak a_n - \mathfrak a_m\in B$ for all $m,n\geq N_0$.  This means that
	\begin{equation*}
		|a_{n,y} - a_{m,y}|_y \leq 1\quad\mbox{for all } m,n\geq N_0\mbox{ and all } y\in Y_E.
	\end{equation*}
	Additionally, there exists a compact set $Z\subseteq Y_E$ containing all archimedean places of $E$ such that $|a_{N_0,y}|_y \leq 1$ for all $y\in Y_E\setminus Z$.
	If $n\geq N_0$ and $y\in Y_E\setminus Z$ then 
	\begin{equation*}
		|a_{n,y}|_y = |a_{n,y} - a_{N_0,y} + a_{N_0,y}|_y \leq \max\{|a_{n,y} - a_{N_0,y}|_y,|a_{N_0,y}|_y\} \leq 1,
	\end{equation*}
	and we have shown that
	\begin{equation} \label{IntegersOutsideZ}
		a_{n,y} \in \mathcal O_y\quad\mbox{for all } n\geq N_0\mbox{ and all } y\in Y_E\setminus Z.
	\end{equation}
	By definition of $\overline{\crazyadele}_E$, the maps $f_{n,v}:Y(E/K,v)\to E_{r_v}$ given by $f_{n,v}(y) = \lambda_v(y,r_v)(a_{n,y})$ are continuous.  
	
	We claim that $f_{n,v}$ converges uniformly for each $v$.  To see this, fix a place $v\in Y_K$, let $\varepsilon >0$, and define
	\begin{equation*}
		U_y = \begin{cases} B_y(0,\varepsilon) & \mbox{if } y\mid v\mbox{ or } y\mid\infty \\ \mathcal O_y & \mbox{if } y\nmid v\mbox{ and } y\nmid\infty. \end{cases}
	\end{equation*}
	We know that $\prod_{y\in Y_E} U_y$ is an open neighborhood of $0$ in $\overline{\crazyadele}_E$ so there exists $M_v\in \nat$ such that
	$|a_{n,y} - a_{m,y}|_y < \varepsilon$ for all $m,n\geq M_v$ and all $y\in Y(E/K,v)$.  Thus, $f_{n,v}$ defines a uniformly Cauchy sequence on $Y(E/K,v)$, and since $E_{r_v}$ is complete, 
	it must converge uniformly to a continuous function $f_v:Y(E/K,v)\to E_{r_v}$.  
	
	For each place $y\in Y_E$ we may assume $v\in Y_K$ is such that $y\mid v$.  Now set $b_y = \lambda_v(r_v,y)(f_v(y))$ and $\mathfrak b = (b_y)_{y\in Y_E}$
	so that $\mathfrak b$ must satisfy \ref{Continuous}.  If $y\in Y_E\setminus Z$ and $n\geq N_0$ then we recall from \eqref{IntegersOutsideZ} that 
	$a_{n,y}\in \mathcal O_y$.  But $a_{n,y}$ converges to $b_y$ in $E_y$ we must have $b_y \in  \mathcal O_y$ as well, and it follows that $\mathfrak b$ satisfies \ref{Compact}.
	It remains to prove that $\mathfrak a_n$ converges to $\mathfrak b$ in $\overline{\crazyadele}_E$.
	
	Let $V = \prod_{y\in Y_E} V_y$ be an open neighborhood of $\mathfrak b$ in $\overline{\crazyadele}_E$.  By \ref{Integers}, there must exist a compact set $Z'$ such that $V_y = \mathcal O_y$ for all
	$y\in Y_E\setminus Z'$.  Certainly $Z\cup Z'$ is compact, so we may select a finite set $S$ of places of $K$ such that $Z\cup Z' \subseteq \bigcup_{v\in S} Y(E/K,v)$.
	For each $y$ not dividing a place in $S$, then we must have $y\in Y_E\setminus Z$, and \eqref{IntegersOutsideZ} implies that $a_{n,y}\in \mathcal O_y$ for all $n\geq N_0$.
	But also $y\in Y_E\setminus Z'$, so we know that $\mathcal O_y = V_y$ and we have found that
	\begin{equation*}
		a_{n,y} \in V_y \quad\mbox{for all } n\geq N_0\mbox{ and all } y\mbox{ not dividing a place in } S.
	\end{equation*}
	Now assume $y$ is a place of $E$ dividing a place $v\in S$.  We apply Lemma \ref{OpenTube} to obtain $\varepsilon_v > 0$ such that 
	$B(f_v(y),\varepsilon_v) \subseteq \lambda(y,r_v)(V_y)$.  Since $f_{n,v}$ converges uniformly to $f_v$ for each $v\in S$, we may select $N_v\in \nat$ such that
	$ f_{n,v}(y) \in B(f_v(y),\varepsilon_v)$ for all $n\geq N_v$ and all $y\in Y(E/K,v)$.  Since $\lambda_v(y,r_v)$ is an isometric isomorphism, we conclude that
	\begin{equation*}
		a_{n,y} \in V_y\quad\mbox{for all } n\geq N_v\mbox{ and all } y\mbox{ dividing a place } v\in S.
	\end{equation*}
	By setting $N =\max\{N_0,\max\{N_v:v\in S\}\}$ we obtain that $\mathfrak a_n\in V$ for all $n\geq N$.
\end{proof}

\section{Proof of Theorem \ref{Main}\ref{CommutativeDiagramMain}} \label{CommutativeDiagramSection}

We continue to assume that $E/F$ is Galois and $K\in \I_E$.  For each point $\mathfrak a = (a_v)_{v\in Y_K} \in \mathbb A_K$ there exists a
unique point $\mathfrak b = (b_y)_{y\in Y_E}\in \mathbb \prod_{y\in Y_E} E_y$ such that $b_y = a_v$ for all $y\in Y(E/K,v)$.  If we select $\lambda_v\in T(E/K,v)$ and
$r_v\in Y(E/K,v)$, then $y\mapsto \lambda_v(y,r_v)(b_y)$ is constant on $Y(E/K,v)$ and its values belong to $\mathcal O_y$ except on a compact subset of $Y_E$.  These observations imply that 
$\mathfrak b\in \overline{\crazyadele}_E$ and we obtain a map $\con_{E/K}:\mathbb A_K\to \overline{\crazyadele}_E$ which is defined so that $\con_{E/K}(\mathfrak a) = \mathfrak b$.
If $L\in \I_E$ is such that $K\subseteq L\subseteq E$ then we have
\begin{equation} \label{ConormCompose}
	\con_{E/L}\circ \con_{L/K} = \con_{E/K}.
\end{equation}
When $E/K$ is a finite extension, then our definition of conorm agrees with the definition provided in the introduction.  
In cases where $E$ is clear from context, we shall often simply write $\crazyadele_K = \con_{E/K}(\mathbb A_K)$.  It follows from \eqref{ConormCompose} that $\crazyadele_K \subseteq \crazyadele_L$
whenever $K\subseteq L$.  We obtain the following theorem which resolves Theorem \ref{Main}\ref{CommutativeDiagramMain}.

\begin{thm} \label{CommutativeDiagramMain2}
	If $E/F$ is a Galois extension and $K\in \I_E$ then $\con_{E/K}$ defines a topological ring isomorphism from $\mathbb A_K$ to $\crazyadele_K$.
	Moreover, the diagram
	\begin{center}
	\begin{tikzcd}[row sep=huge, column sep = huge]
		\mathbb A_L \arrow[r,bend left=10,"\con_{E/L}"]& \crazyadele_L \arrow[l,bend left=10,"\con_{E/L}^{-1}"] \\
		 \mathbb A_K \arrow[u,"\con_{L/K}"] \arrow[r,bend left=10,"\con_{E/K}"] &  \crazyadele_K \arrow[l,bend left=10,"\con_{E/K}^{-1}"] \arrow[u,swap,"inclusion"]
	\end{tikzcd}
	\end{center}
	commutes for all $K,L\in \I_E$ with $K\subseteq L$.
\end{thm}

\begin{proof}
	The commutativity of the diagram follows from \eqref{ConormCompose} and it is trivial to verify that $\con_{E/K}$ is an injective ring homomorphism.  It remains to show that $\con_{E/K}$ is continuous
	and open as a map onto $\crazyadele_K$.  For each place $v$ of $K$, we assume that $\lambda_v\in T(E/K,v)$ and $r_v\in Y(E/K,v)$.
	
	We now prove that $\con_{E/K}$ is continuous.  Assume that $U = \prod_{y\in Y_E} U_y$ is a basis element of $\overline{\crazyadele}_E$ and $\mathfrak a = (a_v)_{v\in Y_K} \in \con_{E/K}^{-1}(U)$.
	For every place $v\in Y_K$ and every place $y\in Y(E/K,v)$, the definition of conorm implies that $a_v\in U_y$, so we conclude that $\lambda_v(y,r_v)(a_v) \subseteq \lambda_v(y,r_v)(U_y)$.
	Using Lemma \ref{OpenTube}, there exists $\varepsilon_v > 0$ such that $B_{r_v}(\lambda_v(y,r_v)(a_v),\varepsilon_v) \subseteq \lambda_v(y,r_v)(U_y)$, and therefore,
	\begin{equation*}
		B_v(a_v,\varepsilon_v) \subseteq B_y(a_v,\varepsilon_v)\subseteq U_y\quad\mbox{for all } y\in Y(E/K,v).
	\end{equation*}
 	We know there exists a finite set $S$ of places of $K$ such that $U_y = \mathcal O_y$ for all $y$ not dividing a place in $S$.  Define
	\begin{equation*}
		\Gamma_v = \begin{cases} B_v(a_v,\varepsilon_v) & \mbox{if } v\in S \\
								\mathcal O_v &\mbox{if } v\not\in S
					\end{cases}
	\end{equation*}
	and $\Gamma = \prod_{v\in Y_K} \Gamma_v$ so that $\Gamma$ is an open set in $\mathbb A_K$ and
	\begin{equation} \label{FinalGamma}
		\mathfrak a \in \Gamma \subseteq \con_{E/K}^{-1}(U)
	\end{equation}
	establishing that $\con_{E/K}$ is continuous.
	
	Now suppose that $\Gamma = \prod_{y\in Y_E} \Gamma_y$ is open in $\mathbb A_K$.  There exists a finite set $S$ of places of $K$ such that $\Gamma_v = \mathcal O_v$ for all $v\in Y_K\setminus S$.
	For $v\in S$, we may assume without loss of generality that
	\begin{equation*}
		\Gamma_v = B_v(c_v,\varepsilon_v)\quad\mbox{for some } c_v\in K_v,\ \varepsilon_v>0.
	\end{equation*} 
	Indeed, sets of this form are a basis for the topology on $\mathbb A_K$. Now define $U= \prod_{y\in Y_E}U_y$ where
	\begin{equation*}
		U_y = \begin{cases} B_{y}(c_v,\varepsilon_v) & \mbox{if } $y$ \mbox{ divides a place } v\in S \\
							\mathcal O_y & \mbox{if } $y$ \mbox{ does not divide a place in } S.
			\end{cases}
	\end{equation*}
	Since $\lambda(y,r_v)(c_v) = c_v$ we know that
	\begin{equation*}
		J(K,\lambda_v,r_v; U) = \bigcup_{y\in Y(E/K,v)} ( \{y\}\times B_{r_v}(c_v,\varepsilon_v)) = Y(E/K,v) \times B_{r_v}(c_v,\varepsilon_v)
	\end{equation*}
	for all $v\in S$.  Certainly this set is open in $Y(E/K,v)\times E_{r_v}$, so it follows that $U$ satisfies properties \ref{Integers} and \ref{Open}.  Additionally, we have that
	$\con_{E/K}(\Gamma) = U \cap \crazyadele_K$ verifying that $\con_{E/K}(\Gamma)$ is open in $\crazyadele_K$.
\end{proof}

\section{Proof of Theorem \ref{Main}\ref{ClosureMain}} \label{ClosureSection}

The proof of Theorem \ref{Main}\ref{ClosureMain} requires two preliminary technical lemmas.

\begin{lem} \label{DisjointFiniteCover}
	Assume $E/F$ is a Galois extension, $K\in \I_E$, and $v\in Y_K$.  For each $y\in Y(E/K,v)$, let $A_y$ be an open neighborhood of $y$.
	Then there exists $L\in \I_E$ with $K\subseteq L$  satisfying the following property: For all $w\in Y(L/K,v)$ there exists $y\in Y(E/K,v)$ such that $Y(E/L,w)\subseteq A_y$.
\end{lem}
\begin{proof}
	For each $y\in Y(E/K,v)$ we may select $K^{(y)}\in \I_E$ and a place $v^{(y)}$ of $K^{(y)}$ such that $y\in Y(E/K^{(y)},v^{(y)}) \subseteq A_y$.
	Therefore, $\{ Y(E/K^{(y)},v^{(y)}): y\in Y(E/K,v)\}$ is an open cover of $Y(E/K,v)$, and by compactness, there exists a finite set $S\subseteq Y(E/K,v)$ such that
	$\{ Y(E/K^{(y)},v^{(y)}): y\in S\}$ is a cover of $Y(E/K,v)$.  Since $E/F$ is Galois, there exists $L\in \I_E$ such that $K^{(y)} \subseteq L$ for all $y\in S$.
	
	If $w\in Y(L/K,v)$ then there exists $y\in S$ such that
	\begin{equation*}
		Y(E/L,w) \cap Y(E/K^{(y)},v^{(y)}) \ne \emptyset.
	\end{equation*}
	We will show that $Y(E/L,w) \subseteq Y(E/K^{(y)},v^{(y)})$.  Assume that $x\in Y(E/L,w) \cap Y(E/K^{(y)},v^{(y)})$.  If $z\in Y(E/L,w)$ then for every $\alpha\in K^{(y)} \subseteq L$ we have
	\begin{equation*}
		|\alpha|_z = |\alpha|_w = |\alpha|_x = |\alpha|_{v^{(y)}}.
	\end{equation*}
	This implies that $z\mid v^{(y)}$ so that $z\in Y(E/K^{(y)},v^{(y)})$.  We now shown that $Y(E/L,w) \subseteq Y(E/K^{(y)},v^{(y)}) \subseteq A_y$ as required.
\end{proof}

\begin{lem} \label{FiniteValues}
	Suppose that $\mathfrak b = (b_y)_{y\in Y_E} \in \overline{\crazyadele}_E$ satisfies the following three properties:
	\begin{enumerate}[label={(\roman*)}]
		\item\label{ZeroCompact} There exists a compact subset $Z\subseteq Y_E$ such that $b_y = 0$ for all $y\in Y_E\setminus Z$
		\item\label{InE} $b_y\in E$ for all $y\in Y_E$
		\item\label{Finite} $\{b_y: y\in Y_E\}$ is finite.
	\end{enumerate}
	Then $\mathfrak b\in \crazyadele_E$.
\end{lem}
\begin{proof}
	Using conditions \ref{InE} and \ref{Finite}, we may assume that $K\in \I_E$ is such that $b_y\in K$ for all $y\in Y_E$.  For each place $v\in Y_K$, let $\lambda_v\in T(E/K,v)$ and $r_v\in Y(E/K,v)$.  
	Further define $f_v: Y(E/K,v)\to K_v$ by $f_v(y) = b_y$, and note that since $\lambda_v(y,r_v)\in \gal(E/K)$, we get
	\begin{equation*}
		f_v(y) = \lambda_v(y,r_v)(b_y).
	\end{equation*}
	We have assumed that $\mathfrak b\in \overline{\crazyadele}_E$ so that $f_v$ is continuous.  However, $\{f_v(y):Y(E/K,v)\}$ is a finite subset
	of the metric space $E_{r_v}$, and therefore, this set is discrete.  In particular, $f_v^{-1}(f_v(y))$ is open in $Y(E/K,v)$ for all $y\in Y(E/K,v)$.
	We may now apply Lemma \ref{DisjointFiniteCover} with $A_y = f_v^{-1}(f_v(y))$.  We obtain $L^{(v)}\in \I_E$ with $K\subseteq L^{(v)}$ such that for all $w\in Y(L^{(v)}/K,v)$ 
	there exists $y_w\in Y(E/K,v)$ such that $Y(E/L^{(v)},w)\subseteq f_v^{-1}(f_v(y_w))$.  It follows that $f_v$ is constant on $Y(E/L^{(v)},w)$ for every place $w\in Y(L^{(v)}/K,v)$.
	
	According to \ref{ZeroCompact}, we may assume that $S$ is a finite set of places of $K$ such that $b_y = 0$ for all $y$ not dividing a place in $S$.
	Now let $L$ be the compositum of the fields $L^{(v)}$, where $v\in S$, so that $L\in \I_E$.  If $u$ is a place of $L$ then $u$ divides some place $v$ of $K$.
	In this case, $Y(E/L,u) \subseteq Y(E/L^{(v)},w)$ for some $w\in Y(L^{(v)}/K,v)$.  If we let $f_u$ be the restriction of $f_v$ to $Y(E/L,u)$, then $f_u$ must be constant and its value must belong to $K$.
	Now define
	\begin{equation*}
		a_u = \begin{cases} f_u(y) &\mbox{if } u\mbox{ divides a place in } S \mbox{ and } y\in Y(E/L,u) \\
			0 & \mbox{if } u\mbox{ does not divide a place in } S
			\end{cases}
	\end{equation*}
	and $\mathfrak a = (a_u)_{u\in Y_L}$ so that $\mathfrak b = \con_{E/L}(\mathfrak a)$.
\end{proof}

Equipped with Lemmas \ref{DisjointFiniteCover} and \ref{FiniteValues} we may proceed with the proof of Theorem \ref{Main}\ref{ClosureMain}.

\begin{proof}[Proof of Theorem \ref{Main}\ref{ClosureMain}]
	Suppose $\mathfrak a = (a_y)_{y\in Y_E}\in \overline{\crazyadele}_E$ and $U = \prod_{y\in Y_E} U_y$ is an open neighborhood of $\mathfrak a$.  We may assume there exists a finite set of places
	$S$ of $F$ containing all archimedean places such that $U_y= \mathcal O_y$ for all $y$ not dividing a place in $S$.
	
	We will find a point $\mathfrak b = (b_y)_{y\in Y_E} \in U$ which satisfies the assumptions of Lemma \ref{FiniteValues}.
	For each $p\in S$, let $\lambda_p\in T(E/F,p)$ and $r_p\in Y(E/F,p)$.  Let $f_p:Y(F,p)\to E_{r_p}$ be given by
	$f_p(y) = \lambda_p(y,r_p)(a_y)$ and note that $f_p$ must be continuous by the definition of $\overline{\crazyadele}_E$.  Then according to Lemma \ref{OpenTube}, there exists $\varepsilon_p >0$ such that
	\begin{equation*}
		B_{r_p}(f_p(y),\varepsilon_p) \subseteq \lambda_p(y,r_p)(U_y)
	\end{equation*}
	for all $y\in Y(E/F,p)$.  Additionally, for each $x\in Y(E/F,p)$, we may select an open set $A_x\subseteq Y(E/F,p)$ containing $x$ such that
	\begin{equation} \label{FirstEpsilon}
		|f_p(y) - f_p(x)|_{r_p} < \frac{\varepsilon_p}{2}\quad\mbox{for all } y\in A_x.
	\end{equation}
	We now apply Lemma \ref{DisjointFiniteCover} with $K = F$ and $v=p$ to obtain $L^{(p)}\in I_E$ such that for all $w\in Y(L^{(p)}/F,p)$ there exists $x_w\in Y(E/F,p)$ such that 
	$Y(E/L^{(p)},w)\subseteq A_{x_w}$.  Since $E$ is dense in $E_{r_p}$ we may let $c_w \in E$ be such that $|f_p(x_w) - c_w|_{r_p} < \varepsilon_p/2$.  Then for each $y\in Y(E/L^{(p)},w)$ we get $y\in A_{x_w}$
	and \eqref{FirstEpsilon} yields
	\begin{equation*}
		|f_p(y) - c_w|_{r_p} = |f_p(y) - f_p(x_w) + f_p(x_w) - c_w|_{r_p} \leq |f_p(y) - f_p(x_w)|_{r_p} + |f_p(x_w) - c_w|_{r_p} < \varepsilon_p.
	\end{equation*}
	In other words, 
	\begin{equation} \label{CContain}
		c_w \in B_{r_p}(f_p(y),\varepsilon_p) \subseteq \lambda_p(y,r_p)(U_y)\quad\mbox{for all } y\in Y(E/L^{(p)},w).
	\end{equation}
	For all places $y\in Y_E$ dividing a place in $S$, we may choose $w$ so that $y\in Y(E/L^{(p)},w)$ and define 
	\begin{equation*}
		b_y = \begin{cases}  \lambda_p(r_p,y)(c_w) &\mbox{if } y\mbox{ divides a place }p\in  S \\
			0 & \mbox{if } y\mbox{ does not divide a place in } S.
			\end{cases}
	\end{equation*}
	Then we set $\mathfrak b = (b_y)_{y\in Y_E}$ and note that $\mathfrak b$ satisfies the assumptions of Lemma \ref{FiniteValues}.  Additionally, \eqref{CContain} implies that $b_y\in U_y$ for all $y\in Y_E$
	and the result follows.
\end{proof}

\section{Proof of Theorem \ref{EqualityEquivalence}} \label{EqualityProof}

Before presenting our proof of Theorem \ref{EqualityEquivalence}, we describe an important situation in which $Y(E/F,p)$ has no isolated points.  We provide a preliminary lemma which accomplishes this goal.

\begin{lem} \label{IsolatedPoints}
	Suppose $E/F$ is an infinite Galois extension and $p$ is a place of $F$.  If there exists a place $r\in Y(E/F,p)$ such that $E_r/F_p$ is algebraic then $Y(E/F,p)$ has no isolated points.
\end{lem}
\begin{proof}
	Since $\gal(E/F)$ acts transitively and continuously on $Y(E/F,p)$, it is enough to show that $r$ is not an isolated point.  Since $E_r/F_p$ is an algebraic extension of complete fields, this extension
	must actually be finite (see \cite[Ch. II, \S 4, Ex. 1]{Neukirch}).
	
	Now assume that $r$ is an isolated point.  Hence, there exists $K\in \I_E$ and a place $v$ of $K$ such that $r$ is the only point belonging to $Y(E/K,v)$.
	Since $E/F$ is an infinite extension, we may define a collection of fields $\{K^{(i)}:i\in \nat_0\} \subseteq \I_E$ such that $K^{(0)} = K$ and $K^{(i)} \subsetneq K^{(i+1)}$ for all $i \in \nat_0$.
	This means that
	\begin{equation} \label{Limit}
		\lim_{i\to \infty} [K^{(i)}:K] = \infty.
	\end{equation}
	If $K^{(i)}$ has two distinct places dividing $v$, then those places must extend to two distinct places of $E$ dividing $v$ contradicting our assumption that $Y(E/K,v)$ contains one point.
	Consequently, for each $i\in \nat_0$, there exists a unique place $v_i$ of $K^{(i)}$ such that $v_i\in Y(K^{(i)}/K,v)$.  However, we note the well-known identity
	\begin{equation*}
		[K^{(i)}:K] = \sum_{w\in Y(K^{(i)}/K,v)} [K^{(i)}_w:K_v]
	\end{equation*}
	which simplifies to $[K^{(i)}:K] = [K^{(i)}_{v_i}:K_v]$ in our case.  These observations yield
	\begin{equation*}
		[K^{(i)}:K] \leq [E_r:K_v] \leq [E_r:F_p].
	\end{equation*}
	However, we have noted that $E_r/F_p$ is finite which contradicts \eqref{Limit}.
\end{proof}

\begin{proof}[Proof of Theorem \ref{EqualityEquivalence}]
	In view of Theorem \ref{Main} and Corollary \ref{DirectLimitCompletionMain}, it is sufficient to show that $\crazyadele_E$ has empty 
	interior in $\overline{\crazyadele}_E$.  Since $\overline{\crazyadele}_E$ is a topological ring and $0\in\crazyadele_E$, it is enough to show that $0$ is not in the interior 
	of $\crazyadele_E$.  To this end, we assume that $U$ is an open neighborhood of $0$ and write
	\begin{equation*}
		U = \prod_{y\in Y_E} U_y.
	\end{equation*}
	We must prove that $U$ contains a point in $\overline{\crazyadele}_E$ which does not belong to $\crazyadele_E$.  There exists a finite set $S$ of places of $F$, containing all archimedean places, such that 
	$U_y = \mathcal O_y$ for all places $y\in Y_E$ not dividing a place in $S$.  We now consider two cases:
	
	{\bf Case 1:} Suppose that $E_y/F_q$ is algebraic for every $q\in Y_F$ and every $y\in Y(E/F,q)$.  In this scenario, we fix places $p \in Y_F\setminus S$ and $r\in Y(E/F,p)$.
		We know that $Y(E/F,p)$ is Hausdorff and compact, and therefore, it is regular. In addition, $Y(E/F,p)$ has a countable basis, which means that the Urysohn Metrization Theorem \cite[Ch. 4, \S 34]{Munkres} 
		applies and shows that $Y(E/F,p)$ is a metric space.  Since $Y(E/F,p)$ is totally disconnected, and according to Lemma \ref{IsolatedPoints} has no isolated points, it is homeomorphic to the Cantor set 
		(see \cite[Corollary 30.4]{Willard}).  Since $p$ is nonarchimedean, $\mathcal O_p$ is also homeomorphic to the cantor set for the same reasons, and consequently, there exists a homeomorphism
		$f: Y(E/F,p)\to \mathcal O_p$.  Then we define $\mathfrak a = (a_y)_{y\in Y_E}$ where
		\begin{equation*}
			a_y = \begin{cases} f(y) & \mbox{if } y\mid p \\ 0 & \mbox{if } y\nmid p. \end{cases}
		\end{equation*}
		Clearly $\mathfrak a$ satisfies property \ref{Compact}.  Assuming that $\lambda\in T(E/F,p)$ and $r\in Y(E/F,p)$, we also know that $\lambda(y,r)(a_y) = f(y)$ for all $y\in Y(E/F,p)$.  
		It now follows that $\mathfrak a$ satisfies \ref{Continuous} and belongs to $\overline{\crazyadele}_E$.  Since $U_y = \mathcal O_y$ for all $y\mid p$ and $0\in U_y$ for all other $y\in Y_E$, 
		we deduce that $\mathfrak a\in U$.
		
		On the other hand, if $\mathfrak a\in \crazyadele_E$ then there exists $K\in \I_E$ such that $\mathfrak a \in \con_{E/K}(\mathbb A_K)$.	
		Therefore, $f$ is constant on $Y(E/K,v)$ for all $v\in Y(K/F,p)$.  However, Lemma \ref{IsolatedPoints}
		shows that $Y(E/F,p)$ cannot have isolated points.  In particular, $Y(E/K,v)$ contains more than one point, contradicting our assumption that $f$ is injective.
	
	{\bf Case 2:} Suppose there exist places $p$ of $F$ and $r\in Y(E/F,p)$ such that $E_r/F_p$ is transcendental.  According to Lemma \ref{OpenTube}, there exists $\varepsilon > 0$ such that
		$B_r(0,\varepsilon) \subseteq \lambda(y,r)(U_y)$ for all $y\in Y(E/F,p)$.
		Since $E_r/F_p$ is transcendental, we may choose a point $\alpha\in E_r$ which does not belong to any finite extension of $F_p$.  By possibly multiplying by an appropriate non-zero element of $F$, 
		we may assume without loss of generality that $\alpha\in B_r(0,\varepsilon)$.  Assume that $\lambda$ is a $p$-adic transition diagram and let 
		$\mathfrak a = (a_y)_{y\in Y_E}$ be such that
		\begin{equation*}	
			a_y = \begin{cases} \lambda(r,y)(\alpha) & \mbox{if } y\mid p \\ 0 & \mbox{if } y\nmid p. \end{cases}
		\end{equation*}
		Clearly $\mathfrak a$ satisfies properties \ref{Compact} and \ref{Continuous} so that $\mathfrak a \in \overline{\crazyadele}_E$.  Also, if $y\mid p$ then 
		\begin{equation*}
			a_y = \lambda(r,y)(\alpha)\in \lambda(r,y)(B_r(0,\varepsilon)) \subseteq \lambda(r,y)(\lambda(y,r)(U_y)) = U_y
		\end{equation*}
		which means that $\mathfrak a\in U$.
		
		If $K\in \I_E$ and $v$ is the place of $K$ such that $r\mid v$ then $\alpha\not\in K_v$ since otherwise $\alpha$ would belong to a finite extension of $F_p$.  
		This means that $\mathfrak a$ cannot belong to $\con_{E/K}(\mathbb A_K)$.  As this argument applies to any $K\in \I_E$, we conclude that $\mathfrak a$
		cannot belong to $\crazyadele_E$.

\end{proof}

\end{document}